\documentclass[reqno]{article}
\usepackage{amssymb}
\usepackage{graphicx}
\usepackage{amsmath}
\usepackage{amsthm}
\usepackage{verbatim}
\usepackage{enumitem}

\usepackage{mathtools}

\newtheorem{thm}{Theorem}
\newtheorem{lem}[thm]{Lemma}

\newtheorem{cor}[thm]{Corollary}
\newtheorem{fact}[thm]{Fact}

\theoremstyle{definition}
\newtheorem{dfn}{Definition}
\newtheorem{rem}{Remark}
\newtheorem{example}{Example}

\mathtoolsset{showonlyrefs}

\newcommand{\re}{{\mathbb R}}
\newcommand{\ce}{{\mathbb C}}
\newcommand{\ze}{{\mathbb Z}}
\newcommand{\te}{{\mathbb T}}
\newcommand{\se}{{\mathbb S}}
\newcommand{\me}{\mathrm{e}}
\newcommand{\mi}{\mathrm{i}}
\newcommand{\md}{\mathrm{d}}

\title{Hydrodynamic Killing vector fields on surfaces}  

\author{Yuuki Shimizu} 

\date{\today} 

\begin{document}

\maketitle
\begin{abstract} 
Killing vector fields, which have their origins in Riemannian geometry, have recently garnered attention for their significance in understanding fluid flows on curved surfaces. 
Owing to the significance of behavior of fluid flows around the boundary and at infinity, in the context of fluid dynamics, Killing vector fields of interest should satisfy the slip boundary condition and be complete vector fields, which are called hydrodynamic Killing vector fields (HKVF) in this paper. 
Our purpose is to determine surfaces admitting a HKVF. 
We prove that any connected, orientable surface admitting an HKVF is conformally equivalent to one of the 14 canonical Riemann surfaces, each with either a rotationally or translationally symmetric metric. 
This paves the way for quantitative investigations of fluid flows associated with Killing vector fields and zonal flows, such as issues of stability and instability, extending its applications potentially to global meteorological phenomena and planetary atmospheric science. 
\end{abstract}

\section{Introduction}
Remarkable fluid flows have been observed on various planets, for instance, Saturn's hexagon and Great Red Spots on Jupiter, not to mention global meteorological phenomena and climate change on a global scale, which have garnered attention from a broad spectrum of researchers from astronomers to geometers. 
Among the myriad of mathematical models that have been proposed to clarify the mechanisms of these flows, in recent years, geometric models focusing on the shape of planets have been proposed. 
For understanding localized vortex structures like Saturn's hexagon, point vortex dynamics offer valuable insights. 
This implies that the existence of a rotational vector field on a sphere leads to an equilibrium state in which point vortices move at a constant speed keeping the polygonal configuration, known as vortex ring~\cite{Dritschel_Boatto_2016}. 
Also for the Great Red Spot, the existence of a rotational vector field on an ellipsoid plays an essential role, that is, a stable zonal flow, defined as a scalar multiple of the rotational vector field, preserves the Great Red Spot~\cite{Tauchi_Yoneda2022_ellipsoid}. 
In general, rotational vector fields on surfaces are considered to be special cases of Killing vector fields, and in recent years there has been a rapid development in fluid dynamics involving Killing vector fields, other than these studies, cf.~\cite{Lichtenfelz_Misiolek_Preston2022,Constantin_Drivas_Ginsberg2021,Miura_2022,Samavaki_Tuomela2020}.

Killing vector fields, defined as vector fields that preserve the Riemannian metric, have traditionally been an object of intensive study, especially within the realms of Riemannian geometry, representation theory, and the theory of Riemannian foliations~\cite{Alexandrino_Bettiol_2015,Berestovskii_Nikonorov_2020}.
In a seemingly unrelated context, the Killing vector fields have recently gained attention in the field of fluid dynamics, arguably due to their role as generalizations of central research subjects like axisymmetric flows and uniform flows.
Additionally, Killing vector fields can be characterized in fluid dynamics as stationary solutions to the Euler-Arnold equation, which describes incompressible and inviscid fluid flows. 
While rotational vector fields on surfaces of revolution and translational vector fields on flat tori have their own significance, translational vector fields on helicoids stand out as particularly significant examples of Killing vector fields that appear in fluid flows within a thin film. 
While identifying Killing vector fields on specific surfaces is relatively straightforward, 
when the discussion is extended to general surfaces serving as fluid flow fields, 
the very existence of non-zero Killing vector fields becomes a less straightforward issue. 
In that regard, it is still unclear for which specific shapes or types of surfaces the insights obtained by assuming the existence of non-zero Killing vector fields are applicable. 

Limited exclusively to closed surfaces, the geometry of surfaces admitting a non-zero Killing vector field is well-investigated.
In fact, no non-zero Killing vector fields exist on compact Riemannian manifolds with negative Ricci curvature everywhere~\cite{Bochner_1946}. 
Furthermore, it is a well-known fact that every orientable closed surface admitting a non-zero Killing vector field is homeomorphic to a sphere or a torus.
These facts, which are applicable only to closed surfaces, could certainly offer insights into fluid dynamics. 
On the other hand, taking into account that rivers are modeled as a channel domain and that behavior near the boundary and at infinity play significant roles in fluid flows in the domain, limiting the scope to just closed surfaces may at times fall short in the field of fluid dynamics. 
Moreover, since solutions for the Euler-Arnold equation, of course, depend on the choice of the Riemannian metric, topological classification alone is not sufficient for examining their quantitative behavior.
In summary, grounded in reasonable assumptions from the perspective of fluid dynamics, it is necessary for applications to clarify the geometry of surfaces admitting a non-zero Killing vector field. 
This paves the way for quantitative investigations of flows associated with Killing vector fields and zonal flows, such as issues of stability and instability, from a more comprehensive perspective.

The purpose of this research is to determine the surfaces admitting a Killing vector field, under appropriate assumptions for fluid dynamics. 
First, we explore the types of Killing vector fields that are suitable in the context of fluid dynamics.
To qualify as steady solutions to the Euler-Arnold equation, Killing vector fields of interest are required to meet the slip boundary condition.
In addition, the Killing vector fields should be complete vector fields in order to avoid any fluid particle leaking out of the surface in finite time. 
Thus, Killing vector fields that are non-zero, complete, and satisfy the slip boundary condition are identified as particularly relevant in the context of fluid dynamics, and are called \textit{Hydrodynamic Killing Vector Fields} (HKVF).

Next, we proceed to classify connected, orientable surfaces that admit a given HKVF.
The surface is conformally equivalent to one of the 14 canonical Riemann surfaces. 
Then, the HKVF is represented by a rotational or translational vector field on the Riemann surface with the Riemannian metric represented as a conformal metric that possesses rotational or translational symmetry, respectively. 
In other words, for any orientable surface that admits an HKVF, we can, without loss of generality, assume that the surface corresponds to one of the 14 canonical Riemann surfaces, each equipped with either a rotationally or translationally symmetric metric.
Our proof is grounded in both the geometric properties of Killing vector fields and the classification of M\"obius transformations on exceptional Riemann surfaces. 
A brief overview of these foundational topics is presented in Section~\ref{sec-prel}. 
For a more in-depth exploration, readers are referred to~\cite{Alexandrino_Bettiol_2015,Farkas_Kra_1991}.

The remainder of this paper is organized as follows.
In Section~\ref{sec-prel}, we briefly review the geometry of Riemann surfaces.
In Section~\ref{sec-hkvf}, we discuss which types of Killing vector fields are of interest in the context of fluid dynamics. 
In Section~\ref{sec-main}, we present the main results, followed by proofs in Section~\ref{sec-proof}.

\section{Preliminaries}
\label{sec-prel}
Let $(M,g)$ be a connected orientable Riemannian $2$-manifold possibly with boundary, called a surface in the present paper. 
Every surface is assumed compact or not. 
As a convention in surface theory, a compact surface without boundary is called a closed surface and a non-compact surface without boundary is called an open surface. 
The infinity in a non-compact surface is described as the notion of ends.
An end of a surface is unbounded connected component of the complement of a compact subset in $M$. 

Let $\mathfrak{X}(M)$ be the space of all vector fields $X:M\to TM$. 
Let $\mathrm{Diff}(M)$ be the space of all diffeomorphisms on $M$. 
For each $X\in\mathfrak{X}(M)$ and some $t\in\re$, we denote the time $t$-map of $X$ by $X_t\in\mathrm{Diff}(M)$. 
A vector field $X\in\mathfrak{X}(M)$ is said to be complete if for each $t\in\re$, $X_t\in\mathrm{Diff}(M)$ is defined. 
Let $\mathrm{Fix}(X)=\{p\in M|\, X(p)=0\}$ and $\mathrm{Per}(X)=\{p\in M|\,\exists  T \in (0,\infty),\, X_T(p)=p\}$. 
Let $\mathcal{L}$ denote the Lie derivative. 
The slip boundary condition for a vector field $X$ on the boundary $\partial M$ is that $g(X,n)=0$ on $\partial M$ where $n$ is the inward unit normal vector for $\partial M$, or equivalently, $X|_{\partial M}\in\mathfrak{X}(\partial M)$. 
The Riemannian metric $g$ and all vector fields and all diffeomorphisms are of class $C^r$ with $r\ge1$. 

Owing to the orientability of $M$, we can utilize techniques from complex analysis.
A complex function $f\colon\mathbb{C}\to \mathbb{C}$ is said to be holomorphic if it satisfies the Cauchy-Riemann equation
\begin{align*}
	\bar\partial f=0.
\end{align*}
Moreover, if $f$ is bijective and its inverse is also holomorphic, then $f$ is called a biholomorphism.
A surface is a complex 1-manifold, or a Riemann surface if the surface is equipped with complex charts $\mathcal{S}=\bigcup_\alpha (U_\alpha,\phi_\alpha)$ that satisfy $M=\bigcup_\alpha U_\alpha$ and for any $(U_\alpha,\phi_\alpha)$, $(U_\beta,\phi_\beta)\in\mathcal{S}$ with $U_\alpha\cap U_\beta\ne \emptyset$, $\phi_\alpha\circ \phi_\beta^{-1} :\phi_\beta( U_\alpha \cap U_\beta) \to \phi_\alpha(U_\alpha\cap U_{\beta}) \subset \mathbb{C} $ is a biholomorphism. 
Two Riemann surfaces $M$ and $N$ are identified if there exists a diffeomorphism $f:M\to N$ such that $f$ is a biholomorphism in each complex chart. 

The surface also has a conformal structure. 
A conformal structure on $M$ is an equivalence class of Riemannian metrics on $M$ under the relation of conformal equivalence.
Two metrics $g$ and $h$ are conformally equivalent if there exists orientation preserving $\varphi\in \mathrm{Diff}(M)$ and a positive smooth function $\lambda$ on $M$ such that 
\begin{align}
	\varphi^*h =\lambda^2 g. 
\end{align}
Every surface $(M,g)$ is locally conformal to the plane, that is, for each $p\in M$, there exists a chart $(U,(x,y))$, called an isothermal chart, centered at $p$ such that the metric has the following local representation in the chart:
\begin{align*}
	g= \lambda^2 (\md x^2+\md y^2).
\end{align*}
Using conventional identification $z=x+\mi y$ for a given isothermal chart $(U,(x,y))$, we obtain a complex chart $(U,z)$.
In particular, $g$ is also written as 
\begin{align}
	g=\lambda^2|\md z|^2
\end{align} 
since 
\begin{align}
	|\md z|^2=\md z\md\bar z=
	(\md x+\mi y)(\md x-\mi y)=\md x^2+\md y^2.
\end{align}
In the present paper, every surface is equipped with a complex structure and a conformal structure. 
We write $M\simeq N$ if $M$ is conformally equivalent to $N$. 
\begin{example}
\label{ex-surf}
	We list some canonical Riemann surfaces of concern in the present paper. 
	Let $\se^1=\re/2\pi\ze$. 
	\begin{enumerate}
		\item The Riemann sphere $\hat\ce=\ce\cup\{\infty\}$.
		\item The complex plane $\ce$.
		\item The unit open disc $\Delta=\{z\in\ce |\, |z|<1\}$, or equivalently, the right half plane $(0,\infty)\times \re$ or the open channel $(0,2\pi)\times \re$.
		\item The punctured plane $\ce^*=\ce\setminus\{0\}$, or equivalently, the cylinder $\se^1\times\re$.
		\item The punctured open disc $\Delta^*=\Delta\setminus\{0\}$. 
		\item An open annulus $\Delta_\rho=\{z\in\ce |\, \rho<|z|<1\}$, $\rho \in(0,1)$. 
		\item A torus $\te_\Gamma=\ce/\Gamma$, $\Gamma=\{m\pi_1+n\pi_2|\,\pi_1,\pi_2\in\ce,\,\mathrm{Im}(\pi_1/\pi_2)>0,\,m,n\in\ze \}$.
		\item The unit closed disc $\overline\Delta=\{z\in\ce|\, |z|\le1\}$ 
		\item The punctured closed disc $\overline\Delta^*=\overline\Delta\setminus\{0\}$. 
		\item A closed annulus $\overline\Delta_\rho=\{z\in\ce|\, \rho\le |z|\le 1\}$. 
		\item A semi-closed annulus $\tilde \Delta_\rho=\{z\in\ce|\, \rho< |z|\le 1\}$.
		\item The closed right half plane $[0,\infty)\times\re$. 
		\item The semi-closed channel $[0,2\pi)\times \re$. 
		\item The closed channel $[0,2\pi]\times \re$.
	\end{enumerate}
	Of theses surface, (1)-(7) are without boundary and (8)-(14) are with boundary. 
	The compact surface in the above list is $\hat\ce$, $\te_\Gamma$, $\overline\Delta$ or $\overline \Delta_\rho$ and the closed surface is the sphere or torus. 
	The open surface is $\ce$, $\Delta$, $\ce^*$, $\Delta^*$ or $\Delta_\rho$.
	
	For (3), (4) and (7), there are several Riemann surfaces that are identified with each other by conformal mappings. 
	We have the conformal mappings $F_p:\Delta\to (0,\infty)\times\re$, $F_h:\Delta\to (0,2\pi)\times \re$, $F_l:\ce^*\to \se^1\times\re$ and their inverse such that 
	\begin{align}
		F_p(z)&= \frac{-z+1}{z+1},
		\quad F_p^{-1}(w)=\frac{w-1}{-w-1},\\
		F_h(z)&=-2\mi \log\left(\mi\frac{1-z}{1+z}\right),
		\quad F_h^{-1}(w)= \frac{1+\mi\me^{\mi w/2}}{1-\mi\me^{\mi w/2}}\\
		F_l(z)&=-\mi \log(z),
		\quad F_l^{-1}(w)=\exp(\mi w).
	\end{align} 
\end{example}

The set of all conformal mappings on $M$ to itself forms a group, called automorphism group $\mathrm{Aut}(M)$. 
If $M$ has no boundary, $\mathrm{Aut}(M)$ implies the conformal structure on $M$ as follows. 
\begin{fact}[{\cite[Theorem in V.4.1]{Farkas_Kra_1991}}]
\label{fact-except-Riem-surf}
	Let $M$ be a surface without boundary.
	Then, $\mathrm{Aut}(M)$ is not a discrete group if and only if $M$ is conformally equivalent to one of the following Riemann surfaces: 
	$\hat\ce$, 
	$\ce$, 
	$\Delta$, 
	$\ce^*$, 
	$\Delta^*$, 
	$\Delta_\rho$, 
	$\te_\Gamma$.
\end{fact}
Every conformal mapping is identified with a biholomorphic mapping, that is a bijective and holomorphic function on $M$ to itself. 
In particular, in the situation stated in Theorem~\ref{fact-except-Riem-surf}, every conformal mapping $f$ is represented as a M\"obius transformation $f(z)=(az+b)/(cz+d)$ for some $a,b,c,d\in\ce$ with $ad-bc=1$. 

\section{Hydrodynamic Killing vector field}
\label{sec-hkvf}
Let us consider what kind of a Killing vector field is reasonable in terms of fluid dynamics on surfaces. 
$X\in\mathfrak{X}(M)$ is said to be Killing if $\mathcal{L}_X g=0$. 
Since
\begin{align}
	\mathcal{L}_X g=\left.\frac{\md }{\md t}X_t^*g\right|_{t=0}=0,
\end{align}
the Killing vector field generates a flow that preserves the Riemannian metric. 
Obviously, if $X\equiv0$, $X$ is a Killing vector field. 
Thus, it should be $X\not \equiv0$ to eliminate the trivial case. 
For instance, in the Euclidean plane, a constant vector field and rotational vector field with constant angular velocity is a Killing vector field. 

The motion of incompressible and inviscid fluids is governed by the Euler-Arnold equation on a surface $(M,g)$:
\begin{align}
	v_t+\nabla_v v=-\mathop{\mathrm{grad}} p,\quad \mathop{\mathrm{div}} v=0
\end{align}
The fluid velocity $v\in\mathfrak{X}(M)$ and the pressure $p\in C(M)$ are unknown at time $t\in (0,T)$. 
If $M$ has boundary, any fluid velocity satisfies the slip boundary condition. 
For each Killing vector field $X\in \mathfrak{X}(M)$, owing to 
\begin{align}
	\nabla_X X=-\mathop{\mathrm{grad}}\mathcal{L}(|X|^2/2),
\end{align}
the Killing vector field $X$ is a steady solution for the Euler-Arnold equation with the pressure given as $p=|X|^2/2$ up to constant as long as it meets the slip boundary condition. 

\begin{figure}[htbp]
  \centering
  \begin{minipage}[b]{.33\hsize}
    \includegraphics[width=\hsize]{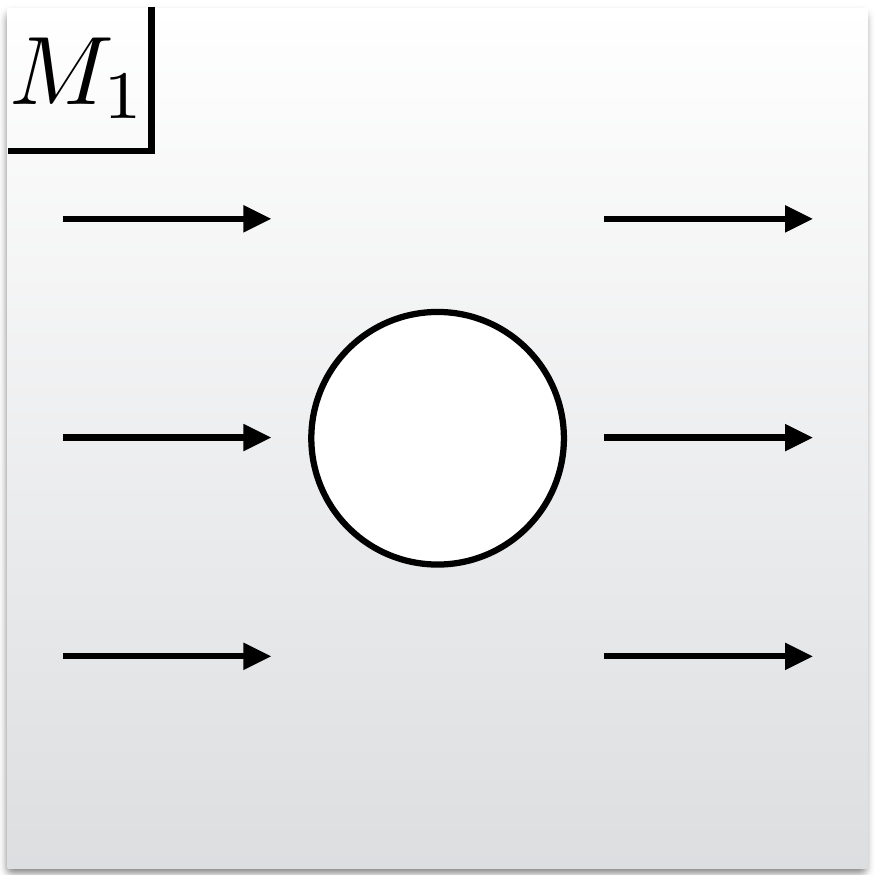}
  \end{minipage}%
  \begin{minipage}[b]{.33\hsize}
    \includegraphics[width=\hsize]{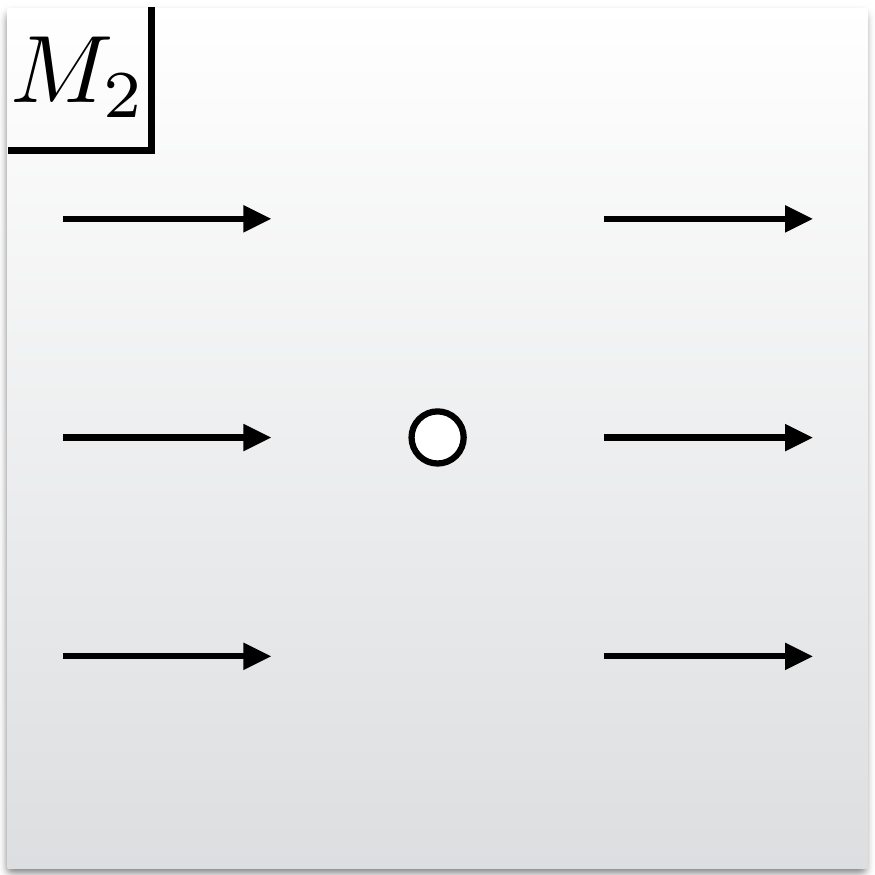}
  \end{minipage}%
  \begin{minipage}[b]{.33\hsize}
    \includegraphics[width=\hsize]{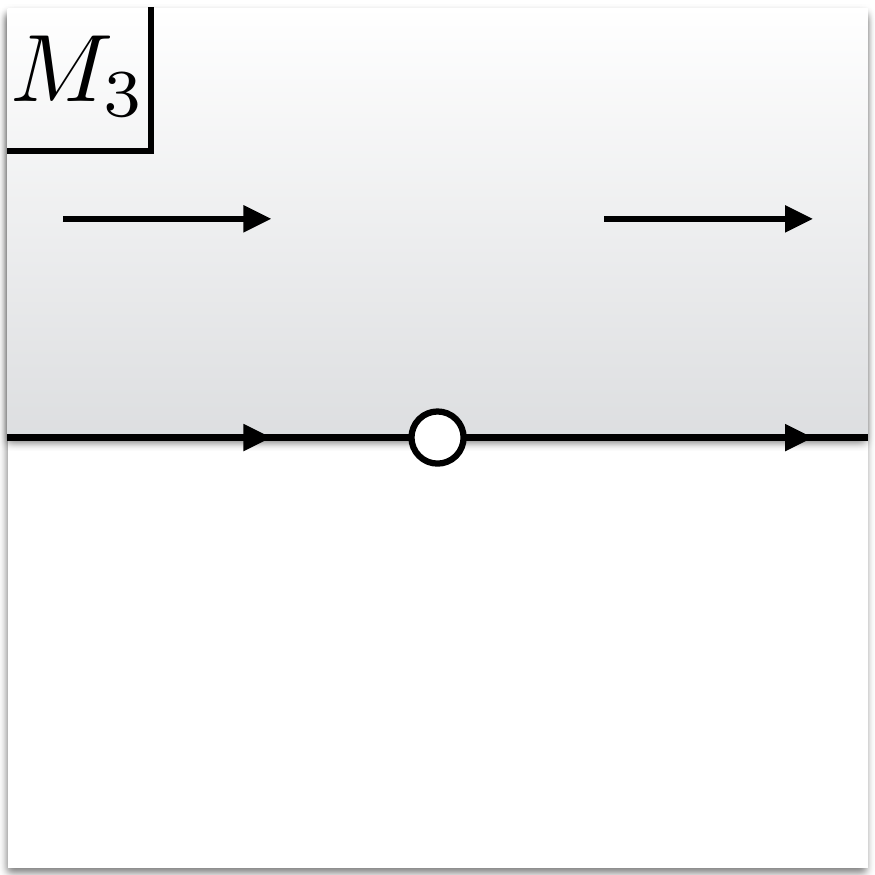}
  \end{minipage}%
  \caption{Examples of non-complete Killing vector fields}
  \label{fig:nc}
\end{figure}

Fig~\ref{fig:nc} describes examples of Killing vector fields on $M_1=\{z\in\ce| \, |z|\ge1\}$, $M_2=\{z\in\ce|\,|z|>0\}$ and $M_3=\{z\in\ce|\, \mathrm{Im} z\geq 0, \,|z|>0\}$ equipped with the Euclidean metric, respectively. 
Of these surfaces, the constant vector field $X=\partial_x$ becomes a Killing vector field. 
In the first example, $X$ does not satisfy the slip boundary condition. 
Concerning the second and third examples, albeit $M_2$ have no boundary, the orbit $\{X_t(p)\}_{t}$ starting at $p=-1$ can not be defined at $t=1$. 
These examples indicates that in order fluid to leak out of these surface, $X$ must satisfy the slip boundary condition and be a complete vector field. 
In addition, the third example implies that the restricted vector field $X|_{\partial M}$ to the boundary also should be complete. 
\begin{dfn}[Hydrodynamic Killing vector field]
	Let $(M,g)$ be a surface. 
	$X\in\mathfrak{X}(M)$ is called \textit{a hydrodynamic Killing vector field} (HKVF for short) if the following conditions are satisfied.
	\begin{enumerate}
		\item $\mathcal{L}_X g=0$.
		\item $X\not\equiv 0$. 
		\item $X|_{\partial M} \in\mathfrak{X}(\partial M)$.
		\item For any $t\in\re$, $X_t\in\mathrm{Diff}(M)$ and $(X|_{\partial M})_t \in\mathrm{Diff}(\partial M)$. 
	\end{enumerate}
\end{dfn}
Any non-zero Killing vector field on a closed surface becomes a HKVF since every vector field on compact manifolds is a complete vector field. 
Clearly, the notion of HKVF is valid even if $\dim M\geq3$. 
On the other hand, taking advantage of $\dim M=2$ and complex analysis, we will classify surfaces with a HKVF. 
It is left to be seen whether a similar classification holds for higher dimensional manifolds with a HKVF. 
\begin{example}
\label{ex-surf}
	Let $R$ be one of the Riemann surfaces listed in Example~\ref{ex-surf} equipped with a flat metric $g$, which gives constant curvature. 
	We define two vector fields $V^{\mathrm{rot}},V^{\mathrm{tra}}\in\mathfrak{X}(R)$ 
	\begin{align}
		V^{\mathrm{rot}}=\partial_\theta,\quad V^{\mathrm{tra}}=\partial_y,
	\end{align} 
	where $z=r\me^{\mi\theta}=x+\mi y\in R$. 
	The time-$t$ map of $V_t^{\mathrm{rot}}$ and $V_t^{\mathrm{rot}}$ are written by
	\begin{align}
		V_t^{\mathrm{rot}}(z)=\me^{\mi t}z,\quad V_t^{\mathrm{tra}}(z)= z+\mi t,
	\end{align}
	which means $V^{\mathrm{rot}}$ and $V^{\mathrm{rot}}$ generate rotation and translation, respectively. 
	Then, $V^{\mathrm{rot}}$ is a HKVF on $R$ if $R= \hat\ce$, $\ce$, $\Delta$, $\ce^*$, $\Delta^*$, $\Delta_\rho$, $\overline\Delta$, $\overline\Delta^*$, $\overline\Delta_\rho$, $\tilde\Delta_\rho$. 
	$V^{\mathrm{tra}}$ is a HKVF on $R$ if $R= \ce$, $(0,\infty)\times \re$, $(0,2\pi)\times \re$, $\se^1\times \re$, $\te_\Gamma$, $[0,\infty)\times\re$, $[0,2\pi)\times \re$, $[0,2\pi]\times \re$. 
	We will show that for every surface, every HKVF on the surface is reduced to one of the above cases via a conformal mapping. 
\end{example}

\section{Main result}
\label{sec-main}
We show that for any surface $(M,g)$, if there exists a hydrodynamic Killing vector field $X\in\mathfrak{X}(M)$, then $M$ is conformally equivalent to one of the canonical Riemann surfaces $R$ listed in Example~\ref{ex-surf} and $X$ is represented as a rotational or translational vector field on $R$. 

\begin{thm}
\label{thm:main}
	Let $(M,g)$ be a surface. 
	Let $X$ be a hydrodynamic Killing vector field on $M$. 
	Then, there exists a Riemann surface $R$ and a conformal mapping $\phi:M\to R$ with the conformal factor $\lambda:R\to (0,\infty)$ 
	such that 
	\begin{enumerate}
		\item if $\mathrm{Per}(X)\ne \emptyset$, $R$ is the one of the following Riemann surfaces: 
		$\hat\ce$, 
		$\ce$, 
		$\Delta$, 
		$\ce^*$, 
		$\Delta^*$, 
		$\Delta_\rho$, 
		$\te_\Gamma$, 
		$\overline\Delta$, 
		$\overline\Delta^*$, 
		$\overline\Delta_\rho$, 
		$\tilde\Delta_\rho$ and
		\begin{align}
			(\phi_*X)_t(z)
			&=
			\begin{cases}
				z+\mi t,&\quad \text{if }R=\te_\Gamma,\\
				\me^{\mi t}z,&\quad\text{otherwise},
			\end{cases}\\
			(\phi^{-1})^*g
			&=
			\begin{cases}
				\lambda^2(\mathrm{Re}(z))|\md z|^2,&\quad \text{if }R=\te_\Gamma,\\
				\lambda^2(|z|)|\md z|^2,&\quad\text{otherwise}.
			\end{cases}
		\end{align}
		\item otherwise, $R$ is the one of the following Riemann surfaces: 
		$\ce$, $(0,\infty)\times \re$, $(0,2\pi)\times \re$, $\se^1\times \re$, $\te_\Gamma$, $[0,\infty)\times\re$, $[0,2\pi)\times \re$, $[0,2\pi]\times \re$
		and 
		\begin{align}
			(\phi_*X)_t(z)
			&=z+\mi t,\\
			(\phi^{-1})^*g
			&=\lambda^2(\mathrm{Re}(z) )|\md z|^2.
		\end{align}
	\end{enumerate}
\end{thm}

\begin{rem}
	When $M\simeq\te_\Gamma$ and when $X$ is a HKVF without periodic point, then $M$ must be a flat torus itself.
	Because, $X$ now conjugates the irrotational rotation on $\te_\Gamma$ and, in consequence, every orbit $\{X_t(p)\}_{t\in\re}$ starting at $p\in M$ is dense in $M$. 
	Since $X_t$ is an isometry, the curvature $K$ is invariant on the orbit of $X$. 
	It follows from the continuity of $K$ that $K$ is constant. 
\end{rem}

By restricting $M$ to $M\setminus\mathrm{Fix}(X)$, 
the following local coordinate (or local trivialization) is available. 
\begin{cor}
\label{cor-E}
	Let $(M,g)$ be a surface with a hydrodynamic Killing vector field $X\in\mathfrak{X}(M)$. 
	Then, there exists a conformal mapping $\phi:M\setminus\mathrm{Fix}(X)\to E$ such that 
	\begin{enumerate}
		\item $E$ is either $B\times \se^1$, $B\times \re$ or $\te_\Gamma$ for some $1$-manifold $B$.
		\item $g$ is represented as $\lambda^2(x^1)((\md x^1)^2+(\md x^2)^2)$ on $E$.
		\item $X$ is represented as $\partial_2$ on $E$. 
	\end{enumerate} 
	In particular, $\lambda=|X|\circ\phi^{-1}$. 
\end{cor}
In other words, for any surface $(M,g)$ and any HKVF $X\in\mathfrak{X}(M)$, taking $\phi:M\setminus\mathrm{Fix}(X)\to E$, without loss of generality, we can assume that on $E$, $g$ and $X$ are given as 
\begin{align}
	g&=\lambda^2(x)(\md x^2+\md y^2),\\
	X&= \frac{\partial}{\partial y}.
\end{align}
for some smooth function $\lambda:E\to (0,\infty)$. 

\section{Proof}
\label{sec-proof}
Theorem~\ref{fact-except-Riem-surf} is helpful for proving the main theorem but it is applicable only for surfaces without boundary. 
It is non-trivial that a conformal mappping can be smoothly extended over the boundary as is a homeomorphism. 
On the other hand, the existence of a HKVF yields that a surface with boundary becomes a smooth submanifold of a surface without boundary.  
Using these properties, we prove the main theorem. 
\begin{lem}
\label{lem:extend}
	Let $(M,g)$ be a surface with boundary. 
	Let $X$ be a HKVF on $M$. 
	Then, there exists an open surface $(\tilde M,\tilde g)$ without boundary and a HKVF $\tilde X\in\mathfrak{X}(\tilde M)$ such that 
	\begin{align}
	\label{eq:extend}
		\tilde M\supset M,\quad \tilde g|_M=g, \quad \tilde X|_{M}=X. 
	\end{align} 
	In particular, any connected component of $\partial M$ is an orbit of $\tilde X$. 
\end{lem}
\begin{proof}
	We first construct a local coordinate centered at $p\in\partial M$ using $X$ and its horizontal geodesic. 
	Let $v\in T_pM$ be the unit inward normal vector. 
	Define $\Pi: (s,t)\in (-\varepsilon,\varepsilon)\times [0,\varepsilon)\to M$ by 
	\begin{align}
		\Pi(s,t)=X_s\circ\gamma(t)
	\end{align} 
	where $\gamma(t)$ is the geodesic $\gamma(t)=\exp_p(tv)$ and sufficiently small $\varepsilon>0$. 
	
	Let us check $\Pi$ is a local diffeomorphism.  
	We have
	\begin{align}
	\label{eq:dPi}
		\Pi_* (\partial_s)
		&= X\circ X_s\circ\gamma(t)
		=(X_s)_* \circ(X_{-s})_*\circ X\circ X_s\circ\gamma(t)
		=(X_s)_* \circ X\circ\gamma(t),\\
		\Pi_* (\partial_t)
		&= (X_s)_* \gamma'(t).
	\end{align}
	Owing to $X$ is a Killing vector field, it follows  
	\begin{align}
		2\frac{\md}{\md t}g_{\gamma(t)}(X\circ\gamma(t),\gamma'(t))
		&=g(\nabla_{\gamma'(t)}X,\gamma'(t))+g(\gamma'(t),\nabla_{\gamma'(t)}X)\\
		&=(\mathcal{L}_Xg)(\gamma'(t),\gamma'(t))\\
		&=0,
	\end{align}
	which gives
	\begin{align}
	\label{eq:hori}
		g_{\gamma(t)}(X\circ\gamma(t),\gamma'(t))
		=g_{\gamma(0)}(X\circ\gamma(0),\gamma'(0))
		=g_p(X(p),v)
		=0.
	\end{align}
	Combining Eq.~\eqref{eq:dPi} and Eq.~\eqref{eq:hori}, we deduce that $\text{rank } \Pi_*=\dim M=2$ and that $\Pi$ is a local diffeomorphism by the inverse function theorem.
	
	Next, we derive a conformal mapping from $\Pi$. 
	As $X_t$ preserves the metric, 
	since
	\begin{align}
		g(\partial_t,\partial_t)
		=X_s^*g(\gamma'(t),\gamma'(t))
		=g(\gamma'(0),\gamma'(0))
		=1, 
	\end{align}
	by Gauss lemma and 
	\begin{align}
		g(\partial_s,\partial_t)
		&=X_s^*g(\gamma'(t),X\circ \gamma(t))
		=g(\gamma'(t),X\circ \gamma(t))
		=0,\\
		g(\partial_s,\partial_s)
		&= X_s^*g(X\circ \gamma(t),X\circ \gamma(t))
		=|X\circ \gamma(t)|^2,
	\end{align}
	we obtain 
	\begin{align}
		\Pi^*g 
		= |X\circ\gamma(t)|^2 \md s^2+ \md t^2
		= |X\circ\gamma(t)|^2(\md s^2+\md f(t)^2)
	\end{align}
	where $f(t)=\int_0^t \frac{\md u }{|X\circ\gamma(u)|}$. 
	Note that the function $f$ is smooth and strictly monotone increasing owing to $|X\circ\gamma(u)|>0$. 
	Hence, $\Sigma: (x,y)\in (-\varepsilon,\varepsilon)\times [0,\varepsilon)\to M$ defined by 
	\begin{align}
		\Sigma(x,y)=X_x\circ\gamma\circ f^{-1}(y)
	\end{align}
	is a conformal mapping, that is, 
	\begin{align}
		\Sigma^*g=|X\circ\gamma\circ f^{-1}(y)|^2(\md x^2+\md y^2). 
	\end{align}
	In particular, 
	\begin{align}
	\label{eq:flowbox}
		\Sigma_*(\partial_x)=X\circ X_x\circ\gamma\circ f^{-1}(y)=X\circ \Sigma(x,y).
	\end{align}
	
	We next establish the desired extension of the surface and the HKVF. 
	For each $p\in\partial M$, we define 
	$\Sigma_p:(x,y)\in U_p \to M$ defined by 
	\begin{align}
		\Sigma(x,y)=X_x\circ\exp_p(f^{-1}(y) v_p) ,
	\end{align}
	where $U_p=(-\varepsilon_p,\varepsilon_p)^2$ for sufficiently small $\varepsilon_p>0$. 
	Let a manifold $\tilde M$ be $\tilde M=M\cup \bigcup_{p\in\partial M} \Sigma_p^{-1}(U_p)$ with compatible local charts $\mathcal{S}_{M}\cup \bigcup_{p\in\partial M} (U_p,\Sigma_p)$. 
	
	Let $\tilde g$ be  
	\begin{align}
		\tilde g_q
		=
		\begin{cases}
			g_q,&\quad\text{if }q\in M,\\
			|X\circ\exp_p(f^{-1}(y) v_p)|^2(\md x^2+\md y^2), &\quad\text{if }q\in \Sigma_p^{-1}(U_p) \text{ for some }p\in\partial M. 
		\end{cases}
	\end{align}
	Let us check that for each $q\in \tilde M\setminus M$, $\tilde g_q$ is determined independently from the choice of $p\in\partial B$.
	If $p_1,p_2\in \partial M$ satisfy $q\in \Sigma_{p_1}(U_{p_1})\cap \Sigma_{p_2}(U_{p_2})$, $q$ is written by 
	\begin{align}
		q=\Sigma_{p_1}(x_1,y_1)=\Sigma_{p_2}(x_2,y_2).
	\end{align}
	Namely, 
	\begin{align}
		X_{x_1}\circ \exp_{p_1}(f^{-1}(y_1)v_{p_1}) = X_{x_2}\circ \exp_{p_2}(f^{-1}(y_2)v_{p_2}).
	\end{align}
	Then, it follows from the naturality of the Riemannian exponential map that 
	\begin{align}
		\exp_{p_2}(f^{-1}(y_2)v_{p_2}) 
		&= X_{x_1-x_2}\circ \exp_{p_1}(f^{-1}(y_1)v_{p_1})\\
		&= \exp_{X_{x_1-x_2}(p_1)}(f^{-1}(y_1)(X_{x_1-x_2})_*v_{p_1}),
	\end{align}
	which yields 
	\begin{align}
		p_2&=X_{x_1-x_2}(p_1).
	\end{align}
	Since $X_t$ is an isometry, we have $v_2=(X_{x_1-x_2})_*v_{p_1}$ and concequently, $f^{-1}(y_2)=f^{-1}(y_1)$. 
	Since $f$ is strictly monotone increasing, we obtain 
	\begin{align}
	\label{eq:gauge}
	\begin{cases}
		x_2=x_1+c,\\
		y_2=y_1 
	\end{cases}
	\end{align}
	for some constant $c$. 
	Hence, $\tilde g$ is well-defined independent from the choice of $p\in\partial M$ and is smooth Riemannian metric on $\tilde M$ satisfying $\tilde g|_M=g$. 
	
	Let us define $\tilde X$ by 
	\begin{align}
		\tilde X_q
		=
		\begin{cases}
			X_q,&\quad\text{if }q\in M,\\
			\partial_x, &\quad\text{if }q\in \Sigma_p^{-1}(U_p) \text{ for some }p\in\partial M. 
		\end{cases}
	\end{align}
	It follows from Eq.~\eqref{eq:gauge} that $\tilde X$ is well-defined independent from the choice of local charts. 
	Owing to Eq.~\eqref{eq:flowbox}, $\tilde X$ is a smooth vector field on $\tilde M$. 
	Since 
	\begin{align}
		\mathcal{L}_{\tilde X}\tilde g
		= \mathcal{L}_{\partial_x} \left[|X\circ\exp_p(f^{-1}(y)v_p)|^2(\md x^2+\md y^2)\right]
		=0,
	\end{align}
	we see that $\tilde X$ is a Killing vector field. 
	By definition of $\tilde X$, we have $\tilde X\not\equiv0$. 
	Since $\tilde M$ has no boundary, $\tilde X$ has nothing to do with the slip boundary condition. 
	Lastly, let us check $\tilde X$ is complete, i.e., for each $q\in \tilde M\setminus M$ and each $t\in\re$, $\tilde X_t(q)$ is well-defined. 
	Choose $p\in\partial M$ such that $q=\exp_p(f^{-1}(y)v_p)$. 
	By completeness of $X$ on $\partial M$, we obtain 
	\begin{align}
		\exp_{X_t(p)}(f^{-1}(y)(X_t)_*v_p)
		=X_t\circ \exp_p(f^{-1}(y)v_p)
		=X_t(q).
	\end{align}
	Hence, $\tilde X$ is a HKVF on $\tilde M$ with $\tilde X|_M=X$. 
	It immediately follows from the construction that each of connected component of $\partial M$ is an orbit   of $\tilde X$. 
\end{proof}

We are now in a position to show the main theorem. 
\begin{proof}[Proof of Theorem~\ref{thm:main}]
	By Lemma~\ref{lem:extend}, the case where $M$ has a boundary can be reduced to the boundaryless case. 
	Therefore, we first make a classification for the boundaryless case and then classify the bounded case based on that classification. 
	As we see in Lemma~\ref{lem:extend}, the bounded surface is then conformally embedded in a boundaryless surfaces, and the boundary is given as an orbit of an extended HKVF. 
	
	Let us focus on the case when the surface has no boundary. 
	We now have an isometric $\re$-action $X_t\in\mathrm{Isom}(M)$ induced by $X$. 
	Since every isometry is a conformal mapping with the conformal factor $\lambda\equiv1$. 
	Hence, $\{X_t\}_{t\in\re}\subset \mathrm{Aut}(M)$ and consequently, $\mathrm{Aut}(M)$ is not a discrete group. 
	We deduce from Theorem~\ref{fact-except-Riem-surf} that $M$ is conformally equivalent to one of boundaryless Riemann surfaces $R=\hat \ce$, $\ce$, $\Delta$, $\ce^*$, $\Delta^*$, $\Delta_\rho$, $\te_\Gamma$ by a conformal mapping $\phi:M\to R$. 
	The induced flow of $\{(\phi_*X)_t\}_{t\in\re}\subset \mathrm{Aut}(R)$ is represented by a M\"obius transformation. 
	Our goal is to find a conformal mapping $\phi:M\to R$ for each canonical Riemann surface $R$ that satisfies either $(\phi_*X)_t(z)=\me^{\mi t}z$ or $(\phi_*X)_t(z)=z+\mi t$. 
	
	Once $\phi$ is found, we see that $(\phi^{-1})^* g=\lambda^2(|z|)|\md z|^2$ if $(\phi_*X)_t(z)=\me^{\mi t}z$ and that $(\phi^{-1})^*g= \lambda^2(\mathrm{Re}(z))|\md z|^2$ if $(\phi_*X)_t(z)=z+\mi t$. 
	Since $X_t$ is an isometry, we have $X_t^*g =g$, which yields 
	\begin{align}
		(\phi_*X)_t^*\circ( \phi^{-1})^* g
		&=(\phi\circ X_t\circ \phi^{-1})\circ (\phi^{-1})^* g\\
		&=(\phi^{-1})^*\circ X_t^*\circ \phi^* \circ (\phi^{-1})^* g\\
		&=(\phi^{-1})^*\circ X_t^*g\\
		&= (\phi^{-1})^*g. \label{eq-phiXg}
	\end{align}
	As $\phi$ is a conformal mapping, the metric $g$ now satisfies 
	\begin{align}
		(\phi^{-1})^* g=\lambda^2|\md z|^2
	\end{align}  
	for some conformal factor $\lambda=\lambda(z,\bar z)$. 
	Since 
	\begin{align}
		(\phi_*X)_t^*\circ( \phi^{-1})^* g
		&=\lambda^2((\phi_*X)_t(z),(\phi_* X)_t(\bar z)) |\md (\phi_*X)_t(z)|^2\\
		&=
		\begin{cases}
			\lambda^2(\me^{\mi t}z,\me^{-\mi t}\bar z)|\md z|^2,&\quad\text{if }(\phi_*X)_t(z)=\me^{\mi t}z,\\
			\lambda^2(z+\mi t,\bar z-\mi t)|\md z|^2,&\quad\text{if }(\phi_*X)_t(z)=z+\mi t,
		\end{cases}
	\end{align}
	it follows from Eq.~\eqref{eq-phiXg} that for each $t\in\re$, 
	\begin{align}
		\lambda(\me^{\mi t}z,\me^{-\mi t}\bar z)=\lambda(z,\bar z)
	\end{align}
	or 
	\begin{align}
		\lambda(z+\mi t,\bar z-\mi t)=\lambda(z,\bar z). 
	\end{align}
	Choosing $t$ suitably for a given $z$, we deduce $\lambda$ depends only on $|z|$ or on $\mathrm{Re}(z)$. 
	
	Let us start to search for the desired conformal mapping $\phi:M\to R$ for each of boundaryless Riemann surfaces $R=\hat \ce$, $\ce$, $\Delta$, $\ce^*$, $\Delta^*$, $\Delta_\rho$, $\te_\Gamma$. 
	\begin{enumerate}
		\item[$\hat\ce$.] 
		Owing to $\mathrm{Aut}(\hat\ce)\simeq PSL(2,\ce)$ and elementary linear algebra, we deduce that for any $f\in \mathrm{Aut}(\hat\ce)$ with $f(\infty)=\infty$, $f$ conjugates to $az+b$ for some $a,b\in\ce$.  
		Hence, there exists a conformal mapping $\phi:M\to \hat\ce $ and time-dependent functions $a:\re\to \ce$ and $b:\re\to \ce$ such that $(\phi_*X)_t(z)=a(t)z+b(t)$.  
		Since $(\phi_*X)_{t+s}=(\phi_*X)_t\circ(\phi_*X)_s$, we obtain 
		\begin{align}
			a(t+s)z+b(t+s)&=a(t)(a(s)z+b(s))+b(t), 
		\end{align}
		which gives 
		\begin{align}
			a(t+s)&=a(t)a(s),\\
			b(t+s)&=a(t)b(s)+b(t).
		\end{align}
		Owing to $(\phi_*X)_0(z)=z$, we now have $a(0)=1$ and $b(0)=0$. 
		Then, from
		\begin{align}
			\frac{a(t+s)-a(t)}{s}&=\frac{a(s)-a(0)}{s}a(t),\\
			\frac{b(t+s)-b(t)}{s}&=\frac{b(s)-b(0)}{s}a(t)
		\end{align}
		and taking these limits as $s\to0$, we derive the following initial value problem for $a(t)$ and $b(t)$: 
		\begin{align}
				\dot a(t)&=\dot a(0) a(t),\quad a(0)=1,\\
				\dot b(t)&=\dot b(0) a(t),\quad b(0)=0.  
		\end{align}
		Solving this, we obtain 
		\begin{align}
			a(t)&=\exp(\dot a(0)t),\\
			b(t)&=
			\begin{cases}
				\dot b(0)t+b(0),\quad &\text{if }\dot a(0)=0,\\
				(\dot b(0)/\dot a(0))\exp(\dot a(0)t),\quad &\text{otherwise}.
			\end{cases}
		\end{align}
		It is necessary for meeting $b(0)=0$ that 
		\begin{align}
			b(t)&=
			\begin{cases}
				\dot b(0)t,\quad &\text{if }\dot a(0)=0,\\
				0,\quad &\text{otherwise}.
			\end{cases}
		\end{align}
		As a result, there exists a conformal mapping $\phi:M\to \hat\ce $ such that $(\phi_*X)_t(z)=\me^{u t}z$ or $z+v t$ for some $u,v\in\ce\setminus\{0\}$. 
		Replacing $t$ by $t/|u|$ or $t/|v|$,without loss of generality, we can assume that $|u|=|v|=1$ and that $\mathrm{Im}(u)$ and $\mathrm{Im}(v)$ are non-negative. 
		
		Let us show that it can be reduced to $(\phi_*X)_t(z)=\me^{\mi t}z$ or $z+\mi t$. 
		Concerning $(\phi_*X)_t(z)=\me^{u t}z$, we can obtain $\mathrm{Re}(u)=0$. 
		If $\mathrm{Re}(u)<0$ were true,   
		$0\in\hat\ce$ would be an attracter of $\phi_*X$. 
		In particular, for each ball $B_r(0)$ centered at $0$ with radius $r$, we have 
		\begin{align}
			\mathrm{Area}((\phi_*X)_t(B_r(0)))
			=\mathrm{Area}(B_{r\me^{\mathrm{Re}\lambda t}}(0))
			\to 0\quad\text{as }t\to \infty,
		\end{align}
		which contradicts $X_t$ is an isometry. 
		In a similar way, we see that $\mathrm{Re}(u) <0$ is not true. 
		Hence, $\mathrm{Re} (u)= 0$. 
		Then, it can be written as $(\phi_*X)_t(z)=\me^{\mi t}z$.  
		
		Concerning $(\phi_*X)_t(z)=z+vt$, 
		we can deduce $(\phi_*X)_t(z)=z+\mi t$ since a rotation $R(z)=\mi z/v$ is a conformal mapping on $\hat\ce$, that is, 
		\begin{align}
			R\circ (\phi_*X)_t\circ R^{-1}=\frac{\mi}{v}\left ( \frac{v}{\mi } z+vt\right)=z+\mi t. 
		\end{align}
		As a result, we deduce either $(\phi_*X)_t(z)=\me^{\mi t}z$ or $(\phi_*X)_t(z)=z+\mi t$. 
		
		If $\mathrm{Per}(X)\ne \emptyset$, $(\phi_*X)_t(z)=\me^{\mi t}z$ holds true.
		Otherwise, we must have $(\phi_*X)_t(z)=z+\mi t$. 
		However, since $\infty\in\hat \ce$ becomes an attracter of $\phi_*X$, every $r$-ball contracts to a point $\infty$, which contradicts $X_t$ is an isometry from the same argument about $\mathrm{Area}((\phi_*X)_t(B_r(0)))$ in the above.
		Hence, we see that there is no case where $\mathrm{Per}(X)= \emptyset$ and that when $M\simeq \hat\ce$, there is only the case for $(\phi_*X)_t(z)=\me^{\mi t}z$. 
		\item[$\ce$.] 
		In the same way as for $\hat\ce$, we can obtain a conformal mapping $\phi:M\to \ce$ that satisfies either $(\phi_*X)_t(z)=\me^{\mi t}z$ or $(\phi_*X)_t(z)=z+\mi t$. 
		If $\mathrm{Per}(X)\ne\emptyset$, we have $(\phi_*X)_t(z)=\me^{\mi t}z$. 
		Otherwise, we obtain $(\phi_*X)_t(z)=z+\mi t$. 
		\item[$\Delta$.] 
		Since $\mathrm{Aut}(\Delta)\simeq PSL(2,\re)$, every non-identical $f\in\mathrm{Aut}(\Delta)$ either has one fixed point in $\Delta$ and the other out of $\Delta$ (elliptic), one fixed point in $\{|z|=1\}$ and nothing else (parabolic) or two fixed points in $\{|z|=1\}$ (hyperbolic). 
		There is no loss of generality in assuming that the set of all fixed points of $f$ is given by either $\{0\}$, $\{-1\}$ or $\{1,-1\}$, respectively. 
		
		For the elliptic case, since $f$ is represented as 
		\begin{align}
			f(z)=\me^{\mi \theta}\frac{z-\alpha}{1-\bar\alpha z},
		\end{align}
		for some $\theta\in \re$ and some $\alpha\in\Delta$, it follows from $f(0)=\alpha$ that $f(z)=\me^{\mi \theta}z$. 
		
		We thus obtain a conformal mapping $\phi:M\to \Delta$ with $(\phi_*X)_t(z)=\me^{\mi \theta(t)}z$ for some time-dependent function $\theta:\re\to \re$. 
		It follows from $(\phi_*X)_{t+s}=(\phi_*X)_{t}\circ(\phi_*X)_{s}$ that 
		\begin{align}
		\label{eq-theta}
			\theta(t+s)=\theta(t)+\theta(s),\quad \theta(0)=0,
		\end{align}
		which gives 
		\begin{align}
			\dot\theta(t)
			&=\lim_{s\to0}\frac{\theta(t+s)-\theta(t)}{s}\\
			&=\lim_{s\to0}\frac{\theta(s)-\theta(0)}{s}\\
			&=\dot\theta(0). 
		\end{align}
		Hence, $\theta(t)=\dot\theta(0) t$. 
		Choosing $t\in\re$ to be $\dot\theta(0)=1$, we deduce $(\phi_*X)_t(z)=\me^{\mi t}z$.
	
		For the parabolic case, we show any $f\in\mathrm{Aut}(\Delta)$ with a multiple fixed point at $z=-1$ conformally conjugates a translation on $(0,\infty)\times\re$ by using the conformal mapping $F_p:\Delta\to (0,\infty)\times\re$, which is defined in Example~\ref{ex-surf}. 
		Let us utilize the following presentation of $f$: 
		\begin{align}
			f(z)=\frac{az+b}{\bar b z+\bar a}
		\end{align}
		for some $a,b\in\ce$ with $|a|^2-|b|^2=1$. 
		In this presentation, the fixed point $z=f(z)$ solves 
		\begin{align}
		\label{eq-pfix}
			\bar b z^2+(\bar a-a)z-b=0.
		\end{align}
		Since the fixed point of $f$ is now only at $z=-1$, Eq.~\eqref{eq-pfix} has a multiple root at $z=-1$. 
		Hence, we deduce
		\begin{align}
			\bar a-a=2\bar b,\quad b+\bar b=0,
		\end{align}
		which gives
		\begin{align}
			a=\pm 1+\mi \theta ,\quad b=\mi \theta
		\end{align}
		for some $\theta\in\re$. 
		Namely, every $f\in\mathrm{Aut}(\Delta)$ with a multiple fixed point at $z=-1$ is presented by
		\begin{align}
			f(z)=f_+(z)\coloneqq\frac{(1+\mi\theta )z+\mi \theta}{-\mi \theta z+1-\mi \theta}
			\quad\text{or}\quad 
			f(z)=f_-(z)\coloneqq\frac{(-1+\mi \theta)z+\mi \theta}{-\mi \theta z-1-\mi \theta}.
		\end{align}
		Then, 
		\begin{align}
			F_p\circ f_+\circ F_p^{-1}=z-2\mi \theta ,\quad 
			F_p\circ f_-\circ F_p^{-1}=z+2\mi \theta.
		\end{align}
		
		Hence, in any case, there exists a conformal mapping $\phi:M\to (0,\infty)\times \re$ such that $(\phi_*X)_t(z)=z+\mi\theta(t)$ with a time-dependent function $\theta:\re\to\re$. 
		Since $(\phi_*X)_{t+s}=(\phi_*X)_t\circ(\phi_*X)_s$, we see that $\theta(t)$ satisfies the same equalities as Eq.~\eqref{eq-theta}. 
		By a similar argument, $\theta(t)=t$ and, in consequence, $(\phi_*X)_t(z)=z+\mi t$.
		
		For the hyperbolic case, we see that for any $f\in\mathrm{Aut}(\Delta)$, if $f$ has fixed points at $z=\pm 1$, $f$ conformally conjugates a translation on $(0,2\pi)\times\re$. 
		In the same manner as the second case, we now have the representation 
		\begin{align}
			f(z)=\frac{az+b}{\bar b z+\bar a}
		\end{align}
		with $a,b\in\ce$ and $|a|^2-|b|^2=1$. 
		Since $f(1)=1$ and $f(-1)=-1$, we obtain $\bar a=a$ and $\bar b=b$. 
		Set 
		\begin{align}
			c=\frac{b}{a}\in(0,1),\quad \theta= -2\log\left(\frac{1-c}{1+c}\right)\in\re.  
		\end{align}
		Then, $f$ is rewritten as 
		\begin{align}
			f(z)=\frac{z+c}{c z+1}. 
		\end{align} 
		Using the conformal mapping $F_h:\Delta\to (0,1)\times \re$, which is defined in Example~\ref{ex-surf}, we have
		\begin{align}
			F_h\circ f\circ F_h^{-1}(z)=z-2\mi \log\left(\frac{1-c}{1+c}\right)= z+\mi \theta.
		\end{align}
		Similarly, we can take a conformal mapping $\phi:M\to (0,2\pi)\times\re$ that satisfies $(\phi_*X)_t(z)=z+\mi t$. 
		
		As a result, if $\mathrm{Per}(X)\ne \emptyset$, $(\phi_*X)_t(z)=\me^{\mi t}z$ on $\Delta$.
		Otherwise, $(\phi_*X)_t(z)=z+\mi t $ on $(0,\infty)\times\re$ or on $(0,1)\times\re$. 	
		\item[$\ce^*$.] 
		Every $f\in\mathrm{Aut}(\ce^*)$ can be extended $\tilde f\in\mathrm{Aut}(\hat\ce)$ that fixes $0$ and $\infty$. 
		In the case for $\hat\ce$, we have seen that there exists a conformal mapping $\phi:M\to \ce^*$ that satisfies either $(\phi_*X)_t(z)=\me^{u t} z$ or $(\phi_*X)_t(z)=z+vt$ with $u,v\in\{z\in\ce|\,|z|=1,\, \mathrm{Im}(z)\geq 0\}$.
		Since $(\phi_*X)_t$ fixes $0$, there is no conformal mapping $\phi:M\to \ce^*$ with $(\phi_*X)_t(z)=z+v t$. 
		Hence, we deduce only the case where $(\phi_*X)_t(z)=\me^{u t} z$.
		
		If $\mathrm{Per}(X)\ne\emptyset$, we have $(\phi_*X)_T(z)=z$ for some $T\in(0,\infty)$, that is, $\me^{u T}=1$ and $\mathrm{Re}(u)=0$. 
		As we see for $\hat \ce$, we can choose $\phi:M\to \ce^*$ such that $(\phi_*X)_t(z)=\me^{\mi t} z$.
		
		If $\mathrm{Per}(X)=\emptyset$, let $F_c: \ce^*\to \se^1\times\re$ be 
		\begin{align}
			F_c(z)= \frac{\mi }{v}\log z.
		\end{align}
		Obviously, $F_c$ is a conformal mapping and $F_c^{-1}(z)=\me^{-\mi v z}$. 
		Since 
		\begin{align}
			F_c\circ (\phi_*X)_t\circ F_c^{-1}(z)= z+\mi t,
		\end{align}
		we obtain the existence of a conformal mapping $\phi:M\to \se^1\times\re$ such that $(\phi_*X)_t=z+\mi t$. 
		\item[$\Delta^*$.]
		In a similar way for $\ce^*$, any $f\in \mathrm{Aut}(\Delta^*)$ has an extension $\tilde f\in\mathrm{Aut}(\Delta)$ that fixes $0\in\Delta$. 
		As we see in the case for $R=\Delta$, when $\tilde f$ has a fixed point at $z=0$, $\tilde f(z)=\me^{\mi \theta}z$. 
		Hence, we deduce $(\phi_*X)_t(z)=\me^{\mi t}z$ on $\Delta^*$.    
		\item[$\Delta_\rho$.]
		Since every $f\in\mathrm{Aut}(\Delta_\rho)$ preserves $\{|z|=1\}$, $f$ is represented as 
		\begin{align}
			f(z)=\me^{\mi \theta}\frac{z-\alpha}{1-\bar\alpha z},
		\end{align}
		for some $\theta\in \re$ and some $\alpha\in\Delta$. 
		We shall derive $\alpha=0$ from the invariance of $\{|z|=\rho\}$ for $f$. 
		Namely, for each $\varphi\in\re$, we now have $|f(\rho\me^{\mi\varphi})|=\rho$. 
		In particular, choosing $\varphi$ that satisfies $\alpha\me^{-\mi \varphi}\in\re$, we obtain 
		\begin{align}
			\left|\frac{\rho-|\alpha|}{1-\rho|\alpha|}\right|=\rho. 
		\end{align}
		Solving this, we deduce $|\alpha|=0$ or $2\rho/(1+\rho^2)$. 
		Owing to $\rho<1$, we must have $\alpha=0$, and consequently, $f=\me^{\mi \theta}z$. 
		Hence, as we see for $\Delta$, we deduce $(\phi_*X)_t(z)=\me^{\mi t}z$ on $\Delta_\rho$. 
		\item[$\te_\Gamma$.]
		Taking the projection $p:\ce\to \te_\Gamma$, for each $f\in\mathrm{Aut}(\te_\Gamma)$, we can take $\tilde f\in\mathrm{Aut}(\ce)$, called a lift, such that 
		\begin{align}
			\label{eq-commute}
			p\circ \tilde f=f\circ p
		\end{align}
		As we see for $\hat\ce$, we now have a conformal mapping $\phi:M\to \te_\Gamma$ that satisfies either $(\phi_*X)_t(z)=\me^{u t} z$ or $(\phi_*X)_t(z)=z+vt$ with $u,v\in\{z\in\ce|\,|z|=1,\, \mathrm{Im}(z)\geq 0\}$.
		
		If $(\phi_*X)_t(z)=\me^{u t} z$ were true, $\phi_*X$ would have a fixed point at $0$. 
		Then, for any $n,m\in\ze$, $m\pi_1+n\pi_2$ also becomes a fixed point, which is impossible. 
		Hence, there is no case where $(\phi_*X)_t(z)=\me^{u t} z$. 
		By using the the rotation $z\mapsto \mi z/v$, we deduce $(\phi_*X)_t(z)=z+\mi t$ on $\te_{\Gamma'}$ with some lattice $\Gamma'$.  
	\end{enumerate}
	Let us summary for the case for surfaces without boundary.
	If $\mathrm{Per}(X)\ne \emptyset$, $M\simeq \hat \ce$, $\ce$, $\Delta$, $\ce^*$, $\Delta^*$, $\Delta_\rho$, $\te_\Gamma$. 
	Then, except for $M\simeq\te_\Gamma$, $(\phi_*X)_t(z)=\me^{\mi t}z$. 
	When $M\simeq\te_\Gamma$, $(\phi_*X)_t(z)=z+\mi t$. 
	If $\mathrm{Per}(X)= \emptyset$, $M\simeq\ce$, $(0,\infty)\times \re$, $(0,2\pi)\times \re$, $\se^1\times \re$, $\te_\Gamma$ and $(\phi_*X)_t(z)= z+\mi t$. 

	We next turn to the case for surfaces with boundary. 
	Lemma~\ref{lem:extend} confirms the existence of an extended open surface $(\tilde M,\tilde g)$ and an extended HKVF $\tilde X\in\mathfrak{X}(\tilde M)$ for a given surface $(M,g)$ with boundary and a given HKVF $X\in\mathfrak{X}(M)$. 
	Since each connected component of $\partial M$ is an orbit of $\tilde X$, 
	the surface is obtained by cutting off a part near ends of open surfaces along a HKVF. 
	
	Let us enumerate all of the possibilities for each of open Riemann surfaces $R=\ce$, $\Delta$, $\ce^*$, $\Delta^*$, $\Delta_\rho$. 
	We now have a conformal mapping $\phi:\tilde M\to R$ such that $(\phi_*\tilde X)_t(z)=\me^{\mi t}z$ or $z+\mi t$. 
	\begin{enumerate}
		\item[$\ce$.] 
		When $(\phi_*\tilde X)_t(z)=\me^{\mi t}z$, every orbit of $\phi_*\tilde X$ is a circle $\{|z|=r\}$ with radius $r>0$, which gives $\phi(\partial M)$. 
		By using the conformal mapping $z\mapsto z/r$, we deduce  $M\simeq\overline \Delta$. 
		
		When $(\phi_*\tilde X)_t(z)=z+\mi t$, every orbit of $\phi_*\tilde X$ is a line $\{\mathrm{Re}(z)=x_1\}$ with $x_1\in\re$. 
		Taking the conformal mapping $z\mapsto z-x_1$, there is no loss of generality by assuming $\phi(M)=[0,\infty)\times\re$. 
		Moreover, cutting off $[0,\infty)\times\re$ along another line $\{\mathrm{Re}(z)=x_2\}$ with $x_2\in[0,\infty)$, we obtain $\phi(M)=[0,x_2]\times \re$. 
		By using $z\mapsto 2\pi z/x_2$, that is reduced to $\phi(M)=[0,2\pi ]\times \re$. 
		As a result, $M\simeq\overline\Delta$, $[0,\infty)\times \re$ or $[0,2\pi ]\times\re$. 
		\item[$\Delta$.] 
		We have seen that if $\mathrm{Per}(\tilde X)\ne \emptyset$, 
		$(\phi_*\tilde X)_t(z)=\me^{\mi t}z$ on $\Delta$ else if $(\phi_* \tilde X)_t(z)=z+\mi t $ on $(0,\infty)\times\re$ or on $(0,2\pi)\times\re$. 
		In the same argument for $\ce$, we deduce that $M\simeq  \overline\Delta$, $[0,\infty)\times \re$, $[0,2\pi)\times\re$ or $[0,2\pi]\times\re$. 
		\item[$\ce^*$.] 
		We then have $(\phi_* \tilde X)_t(z)=\me^{\mi t}z$ on $\ce^*$ or $(\phi_* \tilde X)_t(z)=z+\mi t$ on $\se^1\times \re$. 
		In the former case, cutting off $\ce^*$ along the unit circle, we obtain $\phi(M)=\overline\Delta^*$. 
		In addition, if $\{|z|<\rho\}\subset \overline\Delta^* $ is cut off, $\phi(M)=\overline\Delta_\rho$. 
		
		For the latter case, there is no surface with boundary embedded in this surface,  
		because the boundary curve must be $\{\mathrm{Im}(z)=L\}$ for some $L\in\re$ 
		but then $X=\tilde X|_{M}$ does not satisfy the slip boundary condition.
		Consequently, $M\simeq \overline\Delta^*$, $\overline\Delta_\rho$. 
		\item[$\Delta^*$.] 
		There is only the case $(\phi_* \tilde X)_t(z)=\me^{\mi t}z$. 
		Hence, cutting off $\{|z|>r \}\subset \Delta^* $ with $r\in(0,1)$ and expanding it by $z\mapsto z/r$, we obtain $M\simeq \overline\Delta^*$. 
		\item[$\Delta_\rho$.] 
		Similarly, we only have $(\phi_* \tilde X)_t(z)=\me^{\mi t}z$. 
		Cutting off $\Delta_\rho$ along a circle $\{|z|=1-\varepsilon\}$ and expanding it, we obtain $M\simeq \tilde\Delta_\rho$. 
		In addition, cutting off $\{\rho<|z|<r\}\subset \tilde\Delta_\rho$, and replacing $r$ with $\rho$, we obtain $M\simeq \overline \Delta_\rho$.
		Thus, $M\simeq \tilde \Delta_\rho$ or $\overline \Delta_\rho$. 
	\end{enumerate}
	Summarizing the case when $M$ has boundary, if $\mathrm{Per}(X)\ne\emptyset$, we obtain $(\phi_*X)_t(z)=\me^{\mi t}z$ and $M\simeq \overline \Delta$, $\overline \Delta^*$, $\overline \Delta_\rho$ or $\tilde \Delta_\rho$. 
	Otherwise, it follows $(\phi_*X)_t(z)=z+\mi t$ and $M\simeq [0,\infty)\times\re$, $[0,2\pi)\times \re$ or $[0,2\pi]\times \re$. 
\end{proof}

Lastly, we proceed to prove Corollary~\ref{cor-E}. 
\begin{proof}[Proof of Corollary~\ref{cor-E}]	
	Let us take $\phi:M\to R$ given in Theorem~\ref{thm:main}. 
	When $\mathrm{Per}(X)=\emptyset$ or $M\simeq \te_\Gamma$, 
	the claim has been confirmed in Theorem~\ref{thm:main} 
	and then $\mathrm{Fix}(X)=\emptyset$ and $E=B\times \re$ for some $1$-manifold $B$ or $E=\te_\Gamma$, respectively. 
	
	Concerning the case when $\mathrm{Per}(X)\ne \emptyset$ except for $M\simeq \te_\Gamma$,
	$\log :z=r\me^{\mi \theta}\in R\setminus\{0\}\to \log z= \log r+\mi \theta \in B\times\se^1$ is now a conformal mapping with a certain $1$-manifold $B$. 
	Hence, as long as $\mathrm{Fix}(X)=\emptyset$, $\log\circ\phi: M\to B\times\se^1$ becomes the desired conformal mapping. 
	Because, writing $\log z=x^1+\mi x^2$, we have 
	\begin{align}
		g&=\lambda^2(|z|)|\md z|^2\\
		&= \lambda^2(\me^{x^1})|\me^{x^1+\mi x^2}|^2|\md x^1+\mi \md x^2|^2\\
		&=\left(\lambda(\me^{x^1})\me^{x^1}\right)^2((\md x^1)^2+(\md x^2)^2).
	\end{align}
	Hence, rewriting $\lambda(\me^{x^1})\me^{x^1}$ as $\lambda(x^1)$, we obtain $g=\lambda^2(x^1)( (\md x^1)^2+(\md x^2)^2)$. 
	Furthermore, since
	\begin{align}
		\left((\log\circ\phi)_*X\right)_t(w)
		&=\log\circ( \phi_*X)_t\circ\log^{-1}(w)\\
		&=\log\circ(\phi_*X)_t(\me^{w})\\
		&=\log( \me^{\mi t}\me^{w})\\
		&=w+\mi t, 
	\end{align}
	we deduce $(\log\circ\phi)_*X=\partial_2$. 
	When $\mathrm{Fix}(X)\ne \emptyset$, Theorem~\ref{thm:main} yields that $R=\hat \ce$, $\ce$, $\Delta$, or $\overline\Delta$. 
	In any above case, $\mathrm{Fix}(X)$ is mapped to $0$ or $\infty$ in $R$ by $\phi:M\to R$.
	Hence, in the same argument as the case for $\mathrm{Fix}(X)=\emptyset$, we see that $\log\circ\phi|_{M\setminus\mathrm{Fix}(X)}:M\setminus\mathrm{Fix}(X)\to B\times\se^1$ is the desired conformal mapping. 
	Then, $\lambda =|X|\circ\phi^{-1}$  follows from 
	\begin{align}
		\lambda^2
		&=(\phi^{-1})^*g(\partial_2 ,\partial_2 )\\
		&=g_{\phi^{-1}}(\phi^{-1}_*\partial_2,\phi^{-1}_*\partial_2)\\
		&=g_{\phi^{-1}}(X,X)\\
		&=|X|^2\circ\phi^{-1}. 
	\end{align} 
\end{proof}


\vspace{0.5cm}
\noindent
{\bf Acknowledgments.}\ 
Research of YS was partially supported by Grant-in-Aid for JSPS Fellows 21J00025, Japan Society
for the Promotion of Science (JSPS).
\bibliographystyle{amsplain} 

\begin{thebibliography}{10}
\expandafter\ifx\csname url\endcsname\relax
  \def\url#1{\texttt{#1}}\fi
\expandafter\ifx\csname urlprefix\endcsname\relax\def\urlprefix{URL }\fi
\expandafter\ifx\csname href\endcsname\relax
  \def\href#1#2{#2} \def\path#1{#1}\fi

\bibitem{Dritschel_Boatto_2016}
D.~G. Dritschel, S.~Boatto, \href{https://doi.org/10.1098/rspa.2014.0890}{The motion of point vortices on closed surfaces}, Proc. A. 471~(2176) (2015) 20140890, 25.
\newblock \href {https://doi.org/10.1098/rspa.2014.0890} {\path{doi:10.1098/rspa.2014.0890}}.
\newline\urlprefix\url{https://doi.org/10.1098/rspa.2014.0890}

\bibitem{Tauchi_Yoneda2022_ellipsoid}
T.~Tauchi, T.~Yoneda, \href{https://doi.org/10.2969/jmsj/83868386}{Existence of a conjugate point in the incompressible {E}uler flow on an ellipsoid}, J. Math. Soc. Japan 74~(2) (2022) 629--653.
\newblock \href {https://doi.org/10.2969/jmsj/83868386} {\path{doi:10.2969/jmsj/83868386}}.
\newline\urlprefix\url{https://doi.org/10.2969/jmsj/83868386}

\bibitem{Lichtenfelz_Misiolek_Preston2022}
L.~Lichtenfelz, G.~Misio\l~ek, S.~C. Preston, \href{https://doi.org/10.1093/imrn/rnaa139}{Axisymmetric diffeomorphisms and ideal fluids on {R}iemannian 3-manifolds}, Int. Math. Res. Not. IMRN~(1) (2022) 446--485.
\newblock \href {https://doi.org/10.1093/imrn/rnaa139} {\path{doi:10.1093/imrn/rnaa139}}.
\newline\urlprefix\url{https://doi.org/10.1093/imrn/rnaa139}

\bibitem{Constantin_Drivas_Ginsberg2021}
P.~Constantin, T.~D. Drivas, D.~Ginsberg, \href{https://doi.org/10.1007/s00220-021-04048-4}{Flexibility and rigidity in steady fluid motion}, Comm. Math. Phys. 385~(1) (2021) 521--563.
\newblock \href {https://doi.org/10.1007/s00220-021-04048-4} {\path{doi:10.1007/s00220-021-04048-4}}.
\newline\urlprefix\url{https://doi.org/10.1007/s00220-021-04048-4}

\bibitem{Miura_2022}
T.-H. Miura, Navier-{S}tokes equations in a curved thin domain, {P}art {I}: {U}niform estimates for the {S}tokes operator, J. Math. Sci. Univ. Tokyo 29~(2) (2022) 149--256.

\bibitem{Samavaki_Tuomela2020}
M.~Samavaki, J.~Tuomela, \href{https://doi.org/10.1016/j.geomphys.2019.103543}{Navier-{S}tokes equations on {R}iemannian manifolds}, J. Geom. Phys. 148 (2020) 103543, 15.
\newblock \href {https://doi.org/10.1016/j.geomphys.2019.103543} {\path{doi:10.1016/j.geomphys.2019.103543}}.
\newline\urlprefix\url{https://doi.org/10.1016/j.geomphys.2019.103543}

\bibitem{Alexandrino_Bettiol_2015}
M.~M. Alexandrino, R.~G. Bettiol, \href{https://doi.org/10.1007/978-3-319-16613-1}{Lie groups and geometric aspects of isometric actions}, Springer, Cham, 2015.
\newblock \href {https://doi.org/10.1007/978-3-319-16613-1} {\path{doi:10.1007/978-3-319-16613-1}}.
\newline\urlprefix\url{https://doi.org/10.1007/978-3-319-16613-1}

\bibitem{Berestovskii_Nikonorov_2020}
V.~Berestovskii, Y.~Nikonorov, \href{https://doi.org/10.1007/978-3-030-56658-6}{Riemannian manifolds and homogeneous geodesics}, Springer Monographs in Mathematics, Springer, Cham, [2020] \copyright 2020.
\newblock \href {https://doi.org/10.1007/978-3-030-56658-6} {\path{doi:10.1007/978-3-030-56658-6}}.
\newline\urlprefix\url{https://doi.org/10.1007/978-3-030-56658-6}

\bibitem{Bochner_1946}
S.~Bochner, \href{https://doi.org/10.1090/S0002-9904-1946-08647-4}{Vector fields and {R}icci curvature}, Bull. Amer. Math. Soc. 52 (1946) 776--797.
\newblock \href {https://doi.org/10.1090/S0002-9904-1946-08647-4} {\path{doi:10.1090/S0002-9904-1946-08647-4}}.
\newline\urlprefix\url{https://doi.org/10.1090/S0002-9904-1946-08647-4}

\bibitem{Farkas_Kra_1991}
H.~M. Farkas, I.~Kra, \href{https://doi.org/10.1007/978-1-4612-2034-3}{Riemann surfaces}, 2nd Edition, Vol.~71 of Graduate Texts in Mathematics, Springer-Verlag, New York, 1992.
\newblock \href {https://doi.org/10.1007/978-1-4612-2034-3} {\path{doi:10.1007/978-1-4612-2034-3}}.
\newline\urlprefix\url{https://doi.org/10.1007/978-1-4612-2034-3}

\end{thebibliography}

\end{document}